\documentclass[12pt]{amsart}
\usepackage{amsfonts}
\usepackage{amsmath}
\usepackage{amssymb}
\usepackage[text={5.35in,8.3in}]{geometry}

\newcommand{\ent}{{\mathbb{Z}}}

\newcommand{\real}{{\mathbb{R}}}
\newcommand{\comp}{{\mathbb{C}}}
\newcommand{\quat}{{\mathbb{H}}}

\newcommand{\dist}{{\mathcal{H}}}

\newcommand{\ds}{\displaystyle}

\usepackage{lscape}

\newtheorem{theo}{Theorem}
\newtheorem{coro}{Corollary}
\newtheorem{lemma}{Lemma}
\newtheorem{defi}{Definition}
\newtheorem{prop}{Proposition}

\newcommand{\spn}{{\rm{span}}}

\title[Sub-Riemannian odd dimensional spheres]{Sub-Riemannian geodesics and heat operator on odd dimensional spheres}
\author[Mauricio Godoy M., Irina Markina]{Mauricio Godoy Molina\\ Irina Markina}

\address{Department of Mathematics, University of Bergen, Norway.}
\email{mauricio.godoy@math.uib.no}

\address{Department of Mathematics, University of Bergen, Norway.}
\email{irina.markina@uib.no}

\thanks{The authors are partially supported by the grant of the
Norwegian Research Council \# 177355/V30, by the grant of the
European Science Foundation Networking Programme HCAA and Nordforsk
Research Network ``Analysis and Applications''.}



\begin{document}

\maketitle

\begin{abstract}

In this article we study the sub-Riemannian geometry of the spheres $S^{2n+1}$ and $S^{4n+3}$, arising from the principal $S^1-$bundle structure defined by the Hopf map and the principal $S^3-$bundle structure given by the quaternionic Hopf map respectively. The $S^1$ action leads to the classical contact geometry of $S^{2n+1}$, while the $S^3$ action gives another type of sub-Riemannian structure, with a distribution of corank~3. In both cases the metric is given as the restriction of the usual Riemannian metric on the respective horizontal distributions. For the contact $S^7$ case, we give an explicit form of the intrinsic sub-Laplacian and obtain a commutation relation between the sub-Riemannian heat operator and the heat operator in the vertical direction.

\end{abstract}

\section{Introduction}

One of the main objectives of classical sub-Riemannian geometry is to study manifolds which are path-connected by curves admissible in a certain sense. Admissibility refers to a constraint on the velocity vector of an absolutely continuous curve $\gamma:[0,1]\to M$, where $M$ is a smooth connected manifold. More precisely, if $\dist\subset TM$ is a smooth distribution, then $\gamma$ is admissible or horizontal if $\dot\gamma(t)\in\dist$ a.e. The distribution $\dist$ is often called horizontal distribution in the literature.

The idea of studying sub-Riemannian geometry arising from well-behaved fiber bundles was introduced by R. Montgomery in~\cite{M}, although the Riemannian analogue had been studied many decades before. The idea is the following: given a submersion $\pi:Q\to M$ between two Riemannian manifolds $Q$ and $M$, where $\dim M<\dim Q$, define a ``horizontal'' distribution over $Q$ by the pull-back bundle $\pi^*(TM)$ of the tangent bundle of~$M$ via~$\pi$. In the case when we have a principal $G-$action over $Q$ preserving the fibers of the submersion, the manifold $M$ can be identified with the orbits of the action and, after some technical assumptions, it is possible to obtain an explicit characterization of sub-Riemannian geodesics.

The aim of the present article is to describe the sub-Riemannian geometry of two sub-Riemannian structures for odd-dimensional spheres. More specifically, we study the sub-Riemannian geometry arising from the contact distribution for the spheres $S^{2n+1}$ with metric given as a restriction of the usual Riemannian metric, and the one arising from the quaternionic Hopf fibration for the spheres $S^{4n+3}$.

This article is organized as follows. In Section 2, we give some standard definitions of sub-Riemannian geometry which will be needed in the rest of the paper. In Section 3 we give an explicit description of sub-Riemannian geodesics in spheres $S^{2n+1}$ endowed with the standard contact distribution and we study some of their geometric properties. In Section 4 we use the obtained form of geodesics in the case of $S^3$ to give another interpretation to a result by Hurtado and Rosales in~\cite{HR}. With this new point of view, we are able to extend their result to contact spheres of an arbitrary odd dimension. Section 5 is the analogue to Sections 3 and 4 for the case of spheres of the form $S^{4n+3}$ endowed with a distribution of corank 3. Section 6 is somewhat different technically, but it is in spirit related to the core of this article. It deals with a geodesic differential equation for the quaternionic $\mathbb{H}-$type group studied in~\cite{CM}, obtained generalizing the techniques in~\cite{RitRos}. The reason for studying this equation here is to pose the question of a similar equation for the case of $S^7$ and a distribution of rank 4. Section 7 consists of the construction of the intrinsic sub-Laplacian for $S^7$. The main result states that it is the sum of the squares of an orthonormal basis of the horizontal distribution. Finally, Section 8 employs the previous construction to obtain a simple form of the heat operator for $S^7$ in a similar way as obtained in~\cite{BB}.

\section{Preliminaries and notations}

\subsection{Sub-Riemannian geometry}

Let us first give some general definitions, which will be adapted to our purposes when it will be necessary. Let $M$ be a smooth connected manifold of dimension $n$, together with a smooth distribution $\dist\subset TM$ of rank $k$, $2\leq k<n$. The manifolds of our interest are endowed with distributions satisfying the bracket generating condition, {\em i.e.} distributions whose Lie hull equals the full tangent bundle of $M$. To be more precise, define inductively the vector bundles
$$\dist^1=\dist,\quad\quad\dist^{r+1}=[\dist^r,\dist]+\dist^r\quad\mbox{for }r\geq1,$$
which naturally induce the flag
$$\dist=\dist^1\subseteq\dist^2\subseteq\dist^3\subseteq\ldots.$$
We say that $\dist$ is bracket generating if for all $x\in M$ there is an $r(x)\in\ent^+$ such that
\begin{equation}\label{brgen}
\dist_x^{r(x)}=T_xM.
\end{equation}
If the dimensions $\dim\dist^r_x$ do not depend on $x$ for any $r\geq 1$, we say that $\dist$ is a regular distribution. The least $r$ such that~\eqref{brgen} is satisfied is called the step of $\dist$. In this paper we will focus on regular distributions of step 2.

A natural question to pose is, given $M$ and $\dist$, whether one can join any two points of $M$ via a horizontal curve, {\em i.e.} an absolutely continuous curve $\gamma:[0,1]\to M$ which satisfies $\dot\gamma(t)\in\dist$ almost everywhere. A complete answer to this question was given in \cite{Suss}, which shows a deep generalization the celebrated Chow-Rashevski{\u{\i}} theorem, see \cite{Chow, R}, that gives a sufficient condition and can be stated as follows:

\begin{theo}\label{Chow}
Let $M$ be a connected manifold and $\dist\subset TM$ be a bracket generating distribution, then the set of points that can be connected to $p\in M$ by a horizontal path coincides with $M$.
\end{theo}

\paragraph{\bf{Remark:}} A slightly more general version of Theorem~\ref{Chow} states that, if $M$ is not connected, then the set of points that can be connected to $p\in M$ by a horizontal path is the connected component containing $p$. Since we assumed the manifold to be connected, the general formulation is unnecessary.

After these preliminaries, we are ready to specify the class of manifolds of our interest.

\begin{defi}
A sub-Riemannian structure over a manifold $M$ is a pair $(\dist,\langle\cdot,\cdot\rangle_{sR})$, where $\dist$ is a bracket generating distribution and $\langle\cdot,\cdot\rangle_{sR}$ is a fiber inner product defined on $\dist$. The triple $(M,\dist,\langle\cdot,\cdot\rangle_{sR})$ is called sub-Riemannian manifold.

In this context, the length of a horizontal curve $\gamma:[0,1]\to M$ is defined to be
$$\ell(\gamma):=\int_0^1\|\dot\gamma(t)\|dt,$$
where $\|\dot\gamma(t)\|^2=\langle\dot\gamma(t),\dot\gamma(t)\rangle_{sR}$ whenever $\dot\gamma(t)$ exists.
\end{defi}

This notion of length gives rise to the Carnot-Carath\'eodory distance $d(p,q)$ between two points $p,q\in M$, given by $d(p,q):=\inf\ell(\gamma)$, where the infimum is taken over all absolutely continuous horizontal curves joining $p$ to $q$. An absolutely continuous horizontal curve that realizes the distance between two points is called a horizontal length minimizer. It is clear that if $\dist$ is bracket generating then $d(p,q)$ is a finite nonnegative number.

Considering a trivializing neighborhood $U_p$ around $p\in M$ for the subbundle $\dist$, one can find a local orthonormal basis $X_1,\ldots,X_k$ with respect to $\langle\cdot,\cdot\rangle_{sR}$. The associated sub-Riemannian Hamiltonian is given by
$$H(q,\lambda)=\frac12\sum_{m=1}^k\lambda(X_m(q))^2,$$
where $(q,\lambda)\in T^*U_p$. A normal geodesic corresponds to the projection to $U_p\subset M$ of the solution of the Hamiltonian system
\begin{eqnarray*}
\dot q_i&=&\frac{\partial H}{\partial\lambda_i}\\
\dot\lambda_i&=&-\frac{\partial H}{\partial q_i},
\end{eqnarray*}
where $(q_i,\lambda_i)$ are the coordinates in the cotangent bundle of $M$.

\paragraph{\bf{Remark:}} It is possible to define sub-Riemannian geodesics in a more general context. There are many interesting problems related to the classification of such curves, their analytic and geometric properties. In~\cite{LiuSussmann} the problem for the case of rank two distributions is studied and essentially solved. Nevertheless, in the case of step two distributions, the general notion of geodesic gives rise to two cases: curves consisting of one point and normal geodesics. Thus, normal geodesics are the only interesting case for our purposes. Note that in this case normal geodesics are local length minimizers, in the sense that any sufficiently small arc of a normal geodesic minimizes the length functional. On the other hand one of the particular features of sub-Riemannian geometry, as the sub-Riemannian Heisenberg group exemplifies, is that it is possible to find arbitrarily close points that can be joined by normal geodesics with different lengths.

\subsection{Sub-Riemannian principal bundles}

Our first goal is to recall a full characterization of normal geodesics in the case of sub-Riemannian principal bundles. As a direct application we obtain an explicit formula for the sub-Riemannian geodesics on odd-dimensional spheres, with respect to distributions of corank 1 and 3 in Sections 3 and 5 respectively. For the sake of completeness we recall some definitions and notations given in \cite{M}.

For a submersion $\pi:Q\to M$ with fiber $Q_m=\pi^{-1}(m)$ through $m\in M$, the vertical space at $q\in Q$ is given by $T_qQ_{\pi(q)}$ and it is denoted by $V_q$. In this context, an Ehresmann connection for $\pi:Q\to M$ is a distribution $\dist\subset TQ$ which is everywhere transversal to the vertical space, that is:
$$V_q\oplus\dist_q=T_qQ\quad\mbox{for every }q\in Q.$$

Let us assume that a Lie group $G$ acts on $Q$ in such a way that $\pi:Q\to M$ becomes a fiber bundle with fiber $G$. We say that the submersion $\pi$ is a principal $G-$bundle with connection ${\mathcal{H}}$ if the following conditions hold: $G$ acts freely and transitively on each fiber, the group orbits are the fibers of $\pi:Q\to M$, and the $G-$action on $Q$ preserves the connection ${\mathcal{H}}$. Observe that the second condition implies that $M$ is isomorphic to $Q/G$ and $\pi$ is the canonical projection. We will refer to the connection $\dist$ as the horizontal distribution.

For the rest of this section, let us denote the Lie algebra of $G$ by $\mathfrak{g}$, and the corresponding exponential map by $\exp_G:\mathfrak{g}\to G$.

\begin{defi}
For the principal $G-$bundle $\pi:Q\to M$, the infinitesimal generator for the group action is the map $\sigma_q:\mathfrak{g}\to T_qQ$ defined by $$\sigma_q(\xi)=\left.\frac{d}{d\epsilon}\right|_{\epsilon=0}q\exp_G(\epsilon\xi)$$
for $q\in Q$ and $\xi\in\mathfrak{g}$. If the metric $\langle\cdot,\cdot\rangle$ in $Q$ is $G-$invariant, we have a well-defined bilinear form
$${\mathbb{I}}_q(\xi,\eta)=\langle\sigma_q\xi,\sigma_q\eta\rangle\quad,\quad\xi,\eta\in{\mathfrak{g}},$$
which is called the moment of inertia tensor at $q$.
\end{defi}

The $G-$invariant Riemannian metric on $Q$ is said to be of constant bi-invariant type if its moment of inertia tensor ${\mathbb{I}}_q$ is independent of $q\in Q$. Recall also that, in the case of a principal $G-$bundle, for each $q\in Q$ the infinitesimal generator $\sigma_q$ is an isomorphism between the vertical space $V_q$ and $\mathfrak{g}$. We refer to its inverse as the $\mathfrak{g}$ valued connection one form.

With all of these at hand, we can state the main tool required in this section. This will imply almost immediately Corollaries~\ref{geodScr} and~\ref{geodSquat} which are of core importance in the present paper. The proof of the following theorem can be found in \cite{M}.

\begin{theo}[Horizontal Geodesics for Principal Bundles]\label{horgeod}
Let $\pi:Q\to M$ be a principal $G-$bundle with a Riemannian metric of constant bi-invariant type. Let $\dist$ be the induced connection, with $\mathfrak{g}$ valued connection one form~$A$. Let $\exp_R$ be the Riemannian exponential map, so that $\gamma_R(t)=\exp_R(tv)$ is the Riemannian geodesic through $q$ with velocity vector $v\in T_qQ$. Then any horizontal lift $\gamma$ of the projection $\pi\circ\gamma_R$ is a normal sub-Riemannian geodesic and is given by $$\gamma(t)=\exp_R(tv)\exp_G(-tA(v))$$
where $\exp_G:{\mathfrak{g}}\to G$ is the exponential map of $G$. Moreover, all normal sub-Riemannian geodesics can be obtained in this way.
\end{theo}

\paragraph{\bf{Remark:}} In Theorem~\ref{horgeod}, the sub-Riemannian geodesics are considered with respect to the metric induced by restricting $\langle\cdot,\cdot\rangle$ to $\dist$. Recall that constant bi-invariant metrics must be $G-$invariant.

\section{Sub-Riemannian Geodesics on $S^{2n+1}$}

In the case of odd dimensional spheres $S^{2n+1}$, embedded as the boundary of the unit ball in $\comp^{n+1}$, there is a natural action of $S^1\cong SU(1)$ on it, via componentwise multiplication by a complex number of norm 1. This action induces the well known Hopf fibration $S^1\to S^{2n+1}\to\comp P^n$, which forms a principal $S^1-$bundle with connection $\mathcal{H}$ given by the orthogonal complement to the vector field
\begin{equation}\label{defV}
V_{n+1}(p)=-y_0\partial_{x_0}+x_0\partial_{y_0}-\ldots-y_n\partial_{x_n}+x_n\partial_{y_n}
\end{equation}
at each $p=(x_0,y_0,\ldots,x_n,y_n)\in S^{2n+1}$, with respect to the usual Riemannian metric of $S^{2n+1}$ as embedded in $\real^{2(n+1)}\cong\comp^{n+1}$. In \cite{GM} it is shown that this distribution coincides with the holomorphic tangent space $HS^{2n+1}$ of $S^{2n+1}$ thought as an embedded CR manifold and that it also coincides with the contact distribution given by $\ker\omega$ with respect to the contact form
$$\omega=-y_0dx_0+x_0dy_0-\ldots-y_ndx_n+x_ndy_n.$$
Note that the components of the vector $V_{n+1}(p)$ are the same as in the $\mathfrak{su}(1)$ action $i\cdot p$.

As a direct application of Theorem~\ref{horgeod}, it is possible to describe all sub-Riemannian geodesics for the sphere $S^{2n+1}$ as a sub-Riemannian manifold equipped with connection $\mathcal{H}$ and with metric restricted from $\real^{2(n+1)}$. By the results discussed in~\cite{GM}, the holomorphic tangent space for $S^{2n+1}$ is the distribution induced by the principal $S^1-$bundle given by the Hopf fibration $S^1\to S^{2n+1}\to\comp P^n$ with $\mathfrak{su}(1)-$valued connection form $A(v)=i\langle v,V_{n+1}\rangle$, $v\in T_pS^{2n+1}$, $V_{n+1}$ denotes $V_{n+1}(p)$ and $\langle\cdot,\cdot\rangle$ stands for the standard inner product in $\real^{2(n+1)}$. Moreover, the usual Riemannian structure on $S^{2n+1}$ is of constant bi-invariant type, since we have
$$\left.\frac{d}{d\epsilon}\right|_{\epsilon=0}q\exp_{\mathfrak{su}(1)}(\epsilon\xi)=\alpha i\cdot q=\alpha V_{n+1}(q),$$
for any $q\in S^{2n+1}$ and $\xi=i\alpha\in\mathfrak{su}(1)$. Therefore, the inertia tensor is given by
$$\mathbb{I}_q(i\alpha,i\tilde\alpha)=\langle\alpha V_{n+1}(q),\tilde\alpha V_{n+1}(q)\rangle=\alpha\tilde\alpha,$$
which does not depend of the point.

By Theorem~\ref{horgeod}, we have the following result.
\begin{coro}\label{geodScr}
Let $p\in S^{2n+1}=\{(z_0,\ldots,z_n)\in\comp^{n+1}:|z_0|^2+\ldots+|z_n|^2=1\}$ and $v\in T_pS^{2n+1}$. If $\gamma_R(t)=(z_0(t),\ldots,z_n(t))$ is the great circle satisfying $\gamma_R(0)=p$ and $\dot\gamma_R(0)=v$, then the corresponding sub-Riemannian geodesic is given by
\begin{equation}\label{geodS2n+1}
\gamma(t)=\left(z_0(t)e^{-it\langle v,V_{n+1}\rangle},\ldots,z_n(t)e^{-it\langle v,V_{n+1}\rangle}\right).
\end{equation}
\end{coro}

In order to analyze in more details formula~\eqref{geodS2n+1}, let us introduce some notations and the necessary setup. Recall that the Riemannian geodesic starting at $p\in S^n$ with velocity $v\in T_pS^n$ of any sphere $S^n$ as a submanifold of $\real^{n+1}$, with the standard Riemannian structure, is given by:
\begin{equation}\label{greatcircle}
\gamma_R(t)=p\cos(\|v\|t)+\frac{v}{\|v\|}\sin(\|v\|t),
\end{equation}
where $\|v\|^2=\langle v,v\rangle$. In the case of our interest, a great circle $\gamma_R(t)$ in $S^{2n+1}$ as a submanifold of $\real^{2(n+1)}\cong\comp^{n+1}$ will be written in complex notation as $\gamma_R(t)=(z_0(t),\ldots,z_n(t))$. For notational simplicity, the action of $\lambda\in S^1$ over $(p_0,\ldots,p_n)\in S^{2n+1}$ is denoted by $\lambda\cdot p=(\lambda p_0,\ldots,\lambda p_n)$. Let us write $\gamma(0)=\gamma_R(0)=p=(a_0+ib_0,\ldots,a_n+ib_n)\in S^{2n+1}$ and $\dot\gamma_R(0)=v=(\alpha_0+i\beta_0,\ldots,\alpha_n+i\beta_n)\in T_pS^{2n+1}$. Observe that $V_{n+1}(\gamma(t))=i\cdot\gamma(t)$. As above, $V_{n+1}=V_{n+1}(\gamma(0))$.

\paragraph{\bf{Remark:}} In the subsequent calculations, the notation $\langle\cdot,\cdot\rangle_H$ will denote the standard Hermitian product in $\comp^{n+1}$. We recall that the standard inner product $\langle\cdot,\cdot\rangle$ in $\real^{2(n+1)}$ satisfies $${\rm{Re}}\,\langle\cdot,\cdot\rangle_H=\langle\cdot,\cdot\rangle.$$

Theorem~\ref{horgeod} assures that $\gamma$ is a horizontal curve, {\em i.e.} $\langle\dot\gamma(t),V_{n+1}(\gamma(t))\rangle=0$, nevertheless it is possible to check directly this by straightforward calculations. Since some of the computations will appear later, it is convenient to write them down. First notice that
\begin{eqnarray*}
\langle\dot\gamma(t),V_{n+1}(\gamma(t))\rangle_H&=&\langle(-i\langle v,V_{n+1}\rangle\gamma_R(t)+\dot\gamma_R(t))e^{-i\langle v,V_{n+1}\rangle t},\\
&&ie^{-i\langle v,V_{n+1}\rangle t}\gamma_R(t)\rangle_H\\
&=&-\langle v,V_{n+1}\rangle\langle\gamma_R(t),\gamma_R(t)\rangle_H-i\langle\dot\gamma_R(t),\gamma_R(t)\rangle_H\\
&=&-\langle v,V_{n+1}\rangle-i\langle\dot\gamma_R(t),\gamma_R(t)\rangle_H.
\end{eqnarray*}
Thus the problem is now to determine the value of $$\langle\dot\gamma_R(t),\gamma_R(t)\rangle_H=\sum_{k=0}^n\dot z_k(t)\overline{z_k(t)}.$$
By straightforward calculations, it is easy to see that
\begin{eqnarray}
\sum_{k=0}^n\dot z_k(t)\overline{z_k(t)}&=&(\cos^2(\|v\|t)-\sin^2(\|v\|t))\sum_{k=0}^n(a_k\alpha_k+b_k\beta_k)+\nonumber\\
&&+i\sum_{k=0}^n(a_k\beta_k-b_k\alpha_k)\nonumber\\
&=&\langle p,v\rangle\cos(2\|v\|t)+i\langle v,V_{n+1}\rangle\nonumber\\
&=&i\langle v,V_{n+1}\rangle,\label{hermgammavel}
\end{eqnarray}
yielding to $\langle\dot\gamma(t),V_{n+1}(\gamma(t))\rangle_H=0$, which implies the horizontality of the curve $\gamma(t)$.

Let us now address the problem of connecting two points in $S^{2n+1}$ by sub-Riemannian geodesics. We know by Theorem~\ref{Chow} that it is possible to find a horizontal curve $\Gamma:[0,T]\to S^{2n+1}$ such that
\begin{equation}\label{horS}
\Gamma(0)=p\quad\mbox{and}\quad\Gamma(T)=q,
\end{equation}
for any pair $p,q\in S^{2n+1}$ and all fixed time parameter $T>0$. A natural question to ask is whether $\Gamma$ can be taken as a geodesic in~\eqref{horS}. Due to the complexity of the problem, we will give a partial answer to it. It is important to remark that Proposition~\ref{cpn} is a direct analogue of the result obtained in \cite[Theorem 1]{CMV} in the particular case of $n=1$, {\em i.e.} for the three dimensional sphere.

\begin{prop}\label{cpn}
The set of sub-Riemannian geodesics arising from great circles $\gamma_R(t)$ such that $\dot\gamma_R(0)\in{\mathcal{H}}=\ker\omega$ is diffeomorphic to $\comp P^n$.
\end{prop}

\begin{proof}
In this case any sub-Riemannian geodesic starting at $p\in S^{2n+1}$ with initial velocity $v\in{\mathcal{H}}\subset T_pS^{2n+1}$ coincides with the corresponding great circle, since the condition $\dot\gamma_R(0)\in{\mathcal{H}}= \ker\omega$ is equivalent to $\langle v,V_{n+1}\rangle=0$, thus
$$\gamma(t)=p\cos(\|v\|t)+\frac{v}{\|v\|}\sin(\|v\|t)$$
whose loci is uniquely determined by the point $[v]\in\comp P^n$.
\end{proof}

Observe that this $\comp P^n$ can be seen as a submanifold of $S^{2n+1}$ which is transversal to $V_{n+1}$ along the fiber containing $p$. As remarked in~\cite{CMV} for $S^3$, this can be seen as a sophisticated analogue of the horizontal space at the identity in the $(2n+1)-$dimensional Heisenberg group.

Let us conclude this discussion with an interesting result which will be of importance in the following Section. This can be thought of as a sort of Pythagoras theorem for contact spheres.

\begin{prop}\label{horvel}
For a horizontal sub-Riemannian geodesic of the form
$$\gamma(t)=\left(z_0(t)e^{-it\langle v,V_{n+1}\rangle},\ldots,z_n(t)e^{-it\langle v,V_{n+1}\rangle}\right)$$
the following equation holds
$$\|\dot\gamma(t)\|^2+\langle v, V_{n+1}\rangle^2=\|v\|^2.$$
Thus, its velocity is constant and its sub-Riemannian length for $t\in[a,b]$ is $\ell(\gamma)=(b-a)\sqrt{\|v\|^2-\langle v, V_{n+1}\rangle^2}$.
\end{prop}

\begin{proof}
By straightforward calculations, we have
\begin{eqnarray}
\langle\dot\gamma(t),\dot\gamma(t)\rangle_H&=&\langle(-i\langle v,V_{n+1}\rangle\gamma_R(t)+\dot\gamma_R(t))e^{-i\langle v,V_{n+1}\rangle t}, \nonumber\\ &&(-i\langle v,V_{n+1}\rangle\gamma_R(t)+\dot\gamma_R(t))e^{-i\langle v,V_{n+1}\rangle t}\rangle_H \nonumber\\
&=&\langle v,V_{n+1}\rangle^2\langle\gamma_R(t),\gamma_R(t)\rangle_H+\langle\dot\gamma_R(t),\dot\gamma_R(t)\rangle_H \nonumber\\
&&+\langle v,V_{n+1}\rangle(i\langle\dot\gamma_R,\gamma_R\rangle_H-i\langle\gamma_R,\dot\gamma_R\rangle_H)\nonumber\\
&=&\langle v,V_{n+1}\rangle^2+\|v\|^2-2\langle v,V_{n+1}\rangle^2.\nonumber
\end{eqnarray}
Here we have used equation~\eqref{hermgammavel}. The proposition follows.
\end{proof}

\paragraph{\bf{Remark:}} According to Proposition~\ref{horvel}, the condition that a curve $\gamma(t)=e^{-it\langle v,V_{n+1}\rangle}\gamma_R(t)$ is parameterized by arclength is equivalent to require that $\|v\|^2=1+\langle v,V_{n+1}\rangle^2$.

\section{Curvature of sub-Riemannian geodesics on $S^3$}

In \cite{HR}, the authors describe the horizontal geodesics of the three dimensional sphere with respect to its contact distribution, obtaining an explicit expression for these curves. The key tool to achieve this is the following proposition.

\begin{prop}\label{hurtros}
Let $\gamma:I\to S^3$ be a $C^2$ horizontal curve parameterized by arc-length. Then $\gamma$ is a critical point of length for any admissible variation if and only if there is $\lambda\in\real$ such that $\gamma$ satisfies the second order ordinary differential equation
\begin{equation}
\nabla_{\dot\gamma}\dot\gamma+2\lambda J(\dot\gamma)=0,
\end{equation}
where $\nabla$ is the Levi-Civita connection and $J$ is the standard almost complex structure on $S^3$.
\end{prop}

The authors call the parameter $\lambda$ above the {\em curvature} of $\gamma$, since after projecting it via the Hopf fibration, $\lambda$ becomes precisely the curvature of the projected curve in $S^2$. Note that the curves with zero curvature are precisely the horizontal great circles. It is our purpose to find an explicit expression for $\lambda$ in terms of known parameters of the sub-Riemannian geodesics of $S^3$, as presented in Corollary~\ref{geodScr}.

\begin{prop}\label{expcurv}
The curvature of the sub-Riemannian geodesic
$$\gamma(t)=e^{-i\langle v,V_2\rangle t}\gamma_R(t)$$
in $S^3$, parameterized by arc-length, equals $\langle v,V_2\rangle$.
\end{prop}

\begin{proof}
The Lie group structure of $S^3$ as the set of unit quaternions, induces the globally defined vector fields
\begin{equation}\label{distS3}
\begin{array}{ccl}
V(p)&=&-y_1\partial_{x_1}+x_1\partial_{y_1}-y_2\partial_{x_2}+x_2\partial_{y_2},\\
X(p)&=&-x_2\partial_{x_1}+y_2\partial_{y_1}+x_1\partial_{x_2}-y_1\partial_{y_2},\\
Y(p)&=&-y_2\partial_{x_1}-x_2\partial_{y_1}+y_1\partial_{x_2}+x_1\partial_{y_2},
\end{array}
\end{equation}
at $p=(x_1,y_1,x_2,y_2)\in S^3$, which are orthonormal with respect to the usual Riemannian structure of $\real^3$. Observe that $V(p)=V_2(p)$ as defined in~\eqref{defV}.

Let $p=(x_1,y_1,x_2,y_2)=\gamma(0)\in S^3$ be the initial point of $\gamma$ and let $v=(v_{x_1},v_{y_1},v_{x_2},v_{y_2})=\dot\gamma_R(0)\in T_pS^3$ be the initial velocity of the corresponding great circle. By direct calculation, we have
\begin{equation}\label{velXY}
\dot\gamma(t)=f_X(t)X(\gamma(t))+f_Y(t)Y(\gamma(t)),
\end{equation}
where, denoting $\alpha=\langle v,X\rangle,\beta=\langle v,Y\rangle$, we have
$$f_X(t)=\alpha\cos(2t\langle v,V\rangle)+\beta\sin(2t\langle v,V\rangle),$$
$$f_Y(t)=\beta\cos(2t\langle v,V\rangle)-\alpha\sin(2t\langle v,V\rangle).$$

It follows from this decomposition that
\begin{equation}
J(\dot\gamma(t))=-f_Y(t)X(\gamma(t))+f_X(t)Y(\gamma(t)).
\end{equation}

It remains to determine the term $\nabla_{\dot\gamma}\dot\gamma$. It is well-known that for submanifolds of $\real^n$, the vector field $\nabla_{\dot\gamma}\dot\gamma$ corresponds to the projection of the second derivative $\ddot\gamma$ to the tangent space of the submanifold. In this case, differentiating~\eqref{velXY} we obtain
\begin{eqnarray}
\nabla_{\dot\gamma}\dot\gamma&=&2\langle v,V\rangle(f_Y(t)X(\gamma(t))-f_X(t)Y(\gamma(t)))\nonumber\\
&=&-2\langle v,V\rangle\,J(\dot\gamma(t)).\nonumber
\end{eqnarray}

The proposition follows.
\end{proof}

\paragraph{\bf{Remark:}} Note that in case $p=(1,0,0,0)\in S^3$, a great circle starting at $p$ with velocity vector $v=(0,v_{y_1},v_{x_2},v_{y_2})\in T_pS^3$ is given by
$$\gamma_R(t)=\left(\cos(\|v\|t),\frac{v_{y_1}}{\|v\|}\sin(\|v\|t),\frac{v_{x_2}}{\|v\|}\sin(\|v\|t),\frac{v_{y_2}}{\|v\|}\sin(\|v\|t)\right).$$
Then, the corresponding sub-Riemannian geodesic is
\begin{equation}\label{geodS3}
\gamma(t)=e^{-iv_{y_1}t}\gamma_R(t),
\end{equation}
where $v_{x_2}^2+v_{y_2}^2=1$, since the curve is parameterized by arc-length. It follows that the curvature is given by $\langle v,V_2\rangle=v_{y_1}$.

In~\cite{HR} the problem of existence of closed sub-Riemannian geodesics is also discussed. Their result is that a complete geodesic $\gamma$ in $S^3$ parameterized by arc-length, with curvature $\lambda$ is closed if and only if $\lambda/\sqrt{1+\lambda^2}\in{\mathbb Q}$. This result can be generalized to any odd dimensional sphere.

\begin{prop}\label{clgeod}
Let $\gamma:{\mathbb R}\to S^{2n+1}$ be a complete sub-Riemannian geodesic parameterized by arc-length, with initial velocity $v\in T_pS^{2n+1}$. Then $\gamma$ is closed if and only if $$\frac{\langle v,V_{n+1}\rangle}{\sqrt{1+\langle v,V_{n+1}\rangle^2}}\in{\mathbb Q}.$$
\end{prop}

\begin{proof}
The curve $\gamma:\real\to S^{2n+1}$ is closed if and only if for some $T>0$
$$p=e^{-i\langle v,V_{n+1}\rangle T}\left(p\cos(\|v\|T)+\frac{v}{\|v\|}\sin(\|v\|T)\right).$$

Since $v\in T_pS^{2n+1}$, we know that $v$ is orthogonal to the vector joining $0\in\real^{2n+2}$ to $p$, with respect to the usual Riemannian structure of $\real^{2n+2}$. This means that $\sin(\|v\|T)=0$, which forces $T=k\pi/\|v\|$, $k\in{\mathbb Z}$.

To complete the argument, we only need to see that
$$\pm e^{-ik(\langle v,V_{n+1}\rangle/\|v\|)\pi}p=p$$
if and only if
$$\frac{\langle v,V_{n+1}\rangle}{\|v\|}=\frac{\langle v,V_{n+1}\rangle}{\sqrt{1+\langle v,V_{n+1}\rangle^2}}\in{\mathbb Q},$$
where we have used the remark after Proposition~\ref{horvel}.
\end{proof}

\section{Sub-Riemannian Geodesics on $S^{4n+3}$}\label{gSq}

Let us consider the sphere $S^{4n+3}$ embedded as the boundary of the unit ball in $(n+1)-$dimensional quaternionic space $\quat^{n+1}$. As usual, let us denote the quaternionic units as $i$, $j$, and $k$. There is a natural right action of $Sp(1)\cong S^3$ on $\quat^{n+1}$, via componentwise multiplication by a quaternion of norm one. This action induces a quaternionic Hopf fibrations $S^3\to S^{4n+3}\to\quat P^n$, given by
$$\begin{array}{ccccc}H&:&S^{4n+3}&\to&\quat P^n\\&&(q_0,\ldots,q_n)&\mapsto&[q_0:\ldots:q_n].\end{array}$$

This submersion forms a principal $S^3-$bundle with connection given by the orthogonal complement to the vector fields
$$V_{n+1}^1(p)=-y_0\partial_{x_0}+x_0\partial_{y_0}+w_0\partial_{z_0}-z_0\partial_{w_0}-\ldots-y_n\partial_{x_n}+x_n\partial_{y_n}+w_n\partial_{z_n}-z_n\partial_{w_n},$$
$$V_{n+1}^2(p)=-z_0\partial_{x_0}-w_0\partial_{y_0}+x_0\partial_{z_0}+y_0\partial_{w_0}-\ldots-z_n\partial_{x_n}-w_n\partial_{y_n}+x_n\partial_{z_n}+y_n\partial_{w_n},$$
$$V_{n+1}^3(p)=-w_0\partial_{x_0}+z_0\partial_{y_0}-y_0\partial_{z_0}+x_0\partial_{w_0}-\ldots-w_n\partial_{x_n}-z_n\partial_{y_n}+y_n\partial_{z_n}+x_n\partial_{w_n},$$
at each $p=(x_0,y_0,z_0,w_0\ldots,x_n,y_n,z_n,w_n)\in S^{4n+3}$, with respect to the usual Riemannian metric of $S^{4n+3}$ as embedded in $\real^{4(n+1)}\cong\quat^{n+1}$. It is easy to see that the following commutation relations hold for $V_{n+1}^1,V_{n+1}^2,V_{n+1}^3$
$$[V_{n+1}^1,V_{n+1}^2]=2V_{n+1}^3,\quad[V_{n+1}^2,V_{n+1}^3]=2V_{n+1}^1,\quad[V_{n+1}^1,V_{n+1}^3]=-2V_{n+1}^2.$$
Thus one recovers the fact that $\spn\{V_{n+1}^1(p),V_{n+1}^2(p),V_{n+1}^3(p)\}$ is isomorphic as Lie algebra to $\mathfrak{sp}(1)$, the Lie algebra associated to $S^3$.

It is a well established fact that this distribution is bracket generating. In fact, the geometry of this spheres $S^{4n+3}$ is known to be a quaternionic analogue of CR-geometry, see~\cite{AK}. Note that the components of the vector $V_{n+1}^1(p)$ are the same as in the $\mathfrak{sp}(1)$ action $p\cdot i$. Similar statements hold for $V_{n+1}^2(p)$, $V_{n+1}^3(p)$ and $p\cdot j$, $p\cdot k$ respectively.

In order to apply Theorem~\ref{horgeod} in this situation, it is necessary to specify the $\mathfrak{sp}(1)-$valued connection form associated to the submersion~$H$. In this case, the connection form is given by
$$A(v)=i\langle v,V_{n+1}^1\rangle+j\langle v,V_{n+1}^2\rangle+k\langle v,V_{n+1}^3\rangle$$
where $v\in T_pS^{2n+1}$, $V_{n+1}^\alpha$ denotes $V_{n+1}^\alpha(p)$ ($\alpha=1,2,3$) and $\langle\cdot,\cdot\rangle$ stands for the standard inner product in $\real^{4(n+1)}$.  Moreover, the usual Riemannian structure on $S^{4n+3}$ is of constant bi-invariant type, since for any $q\in S^{4n+3}$ and $\xi=i\alpha+j\beta+k\gamma\in\mathfrak{sp}(1)$, $\alpha,\beta,\gamma\in\real$ we have
\begin{eqnarray*}
\left.\frac{d}{d\epsilon}\right|_{\epsilon=0}q\exp_{\mathfrak{sp}(1)}(\epsilon\xi)&=&\alpha q\cdot i+\beta q\cdot j+\gamma q\cdot k\\
&=&\alpha V_{n+1}^1(q)+\beta V_{n+1}^2(q)+\gamma V_{n+1}^3(q).
\end{eqnarray*}
Therefore, the inertia tensor is given by
$$\mathbb{I}_q(i\alpha+j\beta+k\gamma,i\tilde\alpha+j\tilde\beta+k\tilde\gamma)=$$
$$=\langle\alpha V_{n+1}(q)\beta V_{n+1}^2(q)+\gamma V_{n+1}^3(q),\tilde\alpha V_{n+1}(q)\tilde\beta V_{n+1}^2(q)+\tilde\gamma V_{n+1}^3(q)\rangle=$$
$$=\alpha\tilde\alpha+\beta\tilde\beta+\gamma\tilde\gamma,$$
which does not depend of the point.

As for Corollary~\ref{geodScr}, we have the following result.

\begin{coro}\label{geodSquat}
Let $p\in S^{4n+3}=\{(u_0,\ldots,u_n)\in\quat^{n+1}:|u_0|^2+\ldots+|u_n|^2=1\}$ and $v\in T_pS^{4n+3}$. If $\gamma_R(t)=(u_0(t),\ldots,u_n(t))$ is the great circle satisfying $\gamma_R(0)=p$ and $\dot\gamma_R(0)=v$, then the corresponding sub-Riemannian geodesic is given by
\begin{equation}\label{geodS4n+3}
\gamma(t)=\left(u_0(t)\cdot e^{-tA(v)},\ldots,u_n(t)\cdot e^{-tA(v)}\right).
\end{equation}
\end{coro}

In Corollary~\ref{geodSquat}, the quaternionic exponential is defined by
$$e^{ai+bj+ck}=\cos\sqrt{a^2+b^2+c^2}+\sin\sqrt{a^2+b^2+c^2}\cdot\frac{ai+bj+ck}{\sqrt{a^2+b^2+c^2}},$$
for $a,b,c\in\real$. Note that the curve $e^{-tA(v)}$ is simply the Riemannian geodesic in $S^3$ starting at the identity of the group $e=(1,0,0,0)$, with initial velocity vector $(0,-\langle v,V_{n+1}^1\rangle,-\langle v,V_{n+1}^2\rangle,-\langle v,V_{n+1}^3\rangle)$.

Corollary~\ref{geodSquat} implies immediate analogues to Proposition~\ref{cpn} and to Proposition~\ref{clgeod}, which we state for the sake of completeness. Proofs are adaptations of the aforementioned Propositions.

\begin{prop}
The set of sub-Riemannian geodesics in $S^{4n+3}$ arising from great circles $\gamma_R(t)$ such that $\dot\gamma_R(0)$ is orthogonal to $V_{n+1}^1$, $V_{n+1}^2$ and $V_{n+1}^3$ is diffeomorphic to $\quat P^n$.
\end{prop}

\begin{prop}
Let $\gamma:{\mathbb R}\to S^{4n+3}$ be a complete sub-Riemannian geodesic parameterized by arc-length, with initial velocity $v\in T_pS^{2n+1}$. Then $\gamma$ is closed if and only if $$\frac{\langle v,V_{n+1}^1\rangle}{\|v\|^2},\frac{\langle v,V_{n+1}^2\rangle}{\|v\|^2},\frac{\langle v,V_{n+1}^3\rangle}{\|v\|^2}\in{\mathbb Q}.$$
\end{prop}

In analogy with Proposition~\ref{horvel}, let us consider a similar statement in the case of the spheres $S^{4n+3}$.

\begin{prop}\label{horvelS3}
For a horizontal sub-Riemannian geodesic of the form
$$\gamma(t)=\left(w_0(t)\cdot e^{-tA(v)},\ldots,w_n(t)\cdot e^{-tA(v)}\right)$$
the following equation holds
$$\|\dot\gamma(t)\|^2+\|A(v)\|^2=\|v\|^2,$$
where $\|A(v)\|^2=\langle v,V_{n+1}^1\rangle^2+\langle v,V_{n+1}^2\rangle^2+\langle v,V_{n+1}^3\rangle^2$.
\end{prop}

\begin{proof}

Recall that if $\gamma$ is a sub-Riemannian geodesic, then the length of the velocity vector $\|\dot\gamma(t)\|$ does not depend on $t$. Thus without loss of generality we can assume $t=0$. Let us introduce the following notation
\begin{eqnarray*}
p&=&\gamma(0)=(x_0,y_0,z_0,w_0,\ldots,x_n,y_n,z_n,w_n)\in S^{4n+3},\\
v&=&\dot\gamma_R(0)=(v_{x_0},v_{y_0},v_{z_0},v_{w_0},\ldots,v_{x_n},v_{y_n},v_{z_n},v_{w_n})\in T_pS^{4n+3}.
\end{eqnarray*}
Differentiating equation~\eqref{geodS4n+3} and evaluating at $t=0$, we have
\begin{eqnarray*}
\dot\gamma(0)&=&v-\langle v,V_{n+1}^1\rangle V_{n+1}^1-\langle v,V_{n+1}^2\rangle V_{n+1}^2-\langle v,V_{n+1}^3\rangle V_{n+1}^3.
\end{eqnarray*}

The orthogonality of the vector fields $V_{n+1}^1,V_{n+1}^2,V_{n+1}^3$ implies the desired relation.

\end{proof}

\section{Curvature of sub-Riemannian geodesics on ${\mathbf H}^1$}

The proof of Proposition~\ref{hurtros} is given in~\cite{RitRos} for the case of the three dimensional Heisenberg group. As mentioned in~\cite{HR}, the proof for the case of the sub-Riemannian three dimensional sphere is basically the same. The authors have pointed out, in private communication, that the same result holds for all three dimensional pseudo-Hermitian manifolds.

Note that that if $M$ is either the Heisenberg group of topological dimension 3 or the sphere $S^3$, with Reeb vector field $R$, then the quotient vector bundle
$$TM/\spn\{R\}\to M$$
is trivial. We have not been able to show that the corresponding vector bundle
$$TS^7/\spn\{V_2^1,V_2^2,V_3^3\}\to S^7$$
is trivial, which makes difficult to find an analogous argument to the one employed in~\cite{HR}.

The main goal of this section is to find an analogue to Proposition~\ref{hurtros} for the Gromov-Margulis-Mitchell-Mostow tangent cone of $S^7$, see~\cite{Grom,MM,Mitch,M}, which corresponds to the seven dimensional quaternionic $H-$type group $\mathbf{H}^1$, as presented in~\cite{CM}. Observe that the idea of studying the tangent cone before the sub-Riemannian manifold of interest corresponds to the case in~\cite{RitRos}, since the three dimensional Heisenberg group is the tangent cone to the sub-Riemannian $S^3$. We will study whether this method extends to $S^7$ in a forthcoming paper.

\subsection{The quaternionic $H-$type group $\mathbf{H}^1$}

Let us consider the $4\times 4$ matrices $\mathcal{I,J}$ and $\mathcal K$, given by
$${\mathcal I}=\left(\begin{array}{rrrr}0&1&0&0\\-1&0&0&0\\0&0&0&1\\0&0&-1&0\end{array}\right),\quad{\mathcal J}=\left(\begin{array}{rrrr}0&0&0&-1\\0&0&-1&0\\0&1&0&0\\1&0&0&0\end{array}\right),$$
$${\mathcal K}=\left(\begin{array}{rrrr}0&0&-1&0\\0&0&0&1\\1&0&0&0\\0&-1&0&0\end{array}\right).$$
Note that $\mathcal{I,J}$ and $\mathcal K$ are a fixed representation of the quaternion units, {\em i.e.} if ${\mathcal U}$ denotes the identity matrix of size $4\times4$, then $\spn\{{\mathcal U},{\mathcal I},{\mathcal J},{\mathcal K}\}\cong\quat$ as algebras via the isomorphism
$$\varphi:\spn\{{\mathcal U},{\mathcal I},{\mathcal J},{\mathcal K}\}\to\quat$$
given by $\varphi({\mathcal U})=1,\varphi({\mathcal I})=i,\varphi({\mathcal J})=j,\varphi({\mathcal K})=k$ and extended by linearity.

The seven dimensional quaternionic $H-$type group $\mathbf{H}^1$ corresponds to the manifold $\real^4\oplus\real^3$ with the group operation $\circ$ defined by
$$(x,z)\circ(x',z')=\left(x+x',z_{\mathcal{I}}+z'_{\mathcal{I}}+\frac12\,x'^T{\mathcal{I}}x,\right.$$
$$\left.z_{\mathcal{J}}+z'_{\mathcal{J}}+\frac12\, x'^T{\mathcal{J}}x,z_{\mathcal{K}}+z'_{\mathcal{K}}+
\frac12\, x'^T{\mathcal{K}}x\right)$$
where $x,y,z$ are column vectors and $x'^T,y'^T,z'^T$ are row vectors in $\real^4$.

The Lie algebra $\mathfrak{h}^1$ corresponding to ${\mathbf H}^1$ is spanned by the left invariant vector fields
$$X_1(x,z)=\frac{\partial}{\partial x_1}+\frac{1}{2}\left(+x_2\frac{\partial}{\partial z_{\mathcal I}}-
x_4\frac{\partial}{\partial z_{\mathcal J}}-x_3\frac{\partial}{\partial z_{\mathcal K}}\right),$$
$$X_2(x,z)=\frac{\partial}{\partial x_2}+\frac{1}{2}\left(-x_1\frac{\partial}{\partial z_{\mathcal I}}-
x_3\frac{\partial}{\partial z_{\mathcal J}}+x_4\frac{\partial}{\partial z_{\mathcal K}}\right),$$
$$X_3(x,z)=\frac{\partial}{\partial x_3}+\frac{1}{2}\left(+x_4\frac{\partial}{\partial z_{\mathcal I}}+
x_2\frac{\partial}{\partial z_{\mathcal J}}+x_1\frac{\partial}{\partial z_{\mathcal K}}\right),$$
$$X_4(x,z)=\frac{\partial}{\partial x_4}+\frac{1}{2}\left(-x_3\frac{\partial}{\partial z_{\mathcal I}}+
x_1\frac{\partial}{\partial z_{\mathcal J}}-x_2\frac{\partial}{\partial z_{\mathcal K}}\right),$$
$$Z_{\mathcal I}(x,z)=\frac{\partial}{\partial z_{\mathcal I}},\quad Z_{\mathcal J}(x,z)=\frac{\partial}{\partial z_{\mathcal J}},\quad Z_{\mathcal K}(x,z)=\frac{\partial}{\partial z_{\mathcal K}}.$$
at a point $(x,z)=(x_1,x_2,x_3,x_4,z_{\mathcal I},z_{\mathcal J},z_{\mathcal K})\in\mathbf{H}^1$. A Riemannian metric $\langle\cdot,\cdot\rangle$ in $\mathbf{H}^1$ is declared so that $X_1,\ldots,X_4,Z_{\mathcal I},\ldots,Z_{\mathcal K}$ is an orthonormal frame at each $(x,z)\in\mathbf{H}^1$. The sub-Riemannian structure on ${\mathbf H}^1$ we are interested in is defined by the left invariant distribution ${\mathcal D}=\spn\{X_1,X_2,X_3,X_4\}$ and the restriction of the metric previously defined.

Observe that ${\mathcal D}$ is bracket generating of step two. In fact, we have the commutator relations
\begin{equation}\label{commHtype}
\begin{array}{ccccr}
\mbox{$[X_1,X_2]$}&=&[X_3,X_4]&=&-Z_{\mathcal I},\\
\mbox{$[X_2,X_3]$}&=&[X_1,X_4]&=&Z_{\mathcal J},\\
\mbox{$[X_1,X_3]$}&=&[X_4,X_2]&=&Z_{\mathcal K}.
\end{array}
\end{equation}
All the remaining commutators between the chosen basis of $\mathfrak{h}^1$ vanish.

From the well-known Koszul formula for the Levi-Civita connection associated to the metric $\langle\cdot,\cdot\rangle$
$$\langle Z,\nabla_YX\rangle=\frac12(X\langle Y,Z\rangle+Y\langle Z,X\rangle-Z\langle X,Y\rangle-$$ $$-\langle[X,Z],Y\rangle-\langle[Y,Z],X\rangle-\langle[X,Y],Z\rangle),$$
see for example~\cite{dC}, the orthonormality of the basis $\{X_1,\ldots,X_4,Z_{\mathcal I},\ldots,Z_{\mathcal K}\}$, and equations~\eqref{commHtype} we get that
$$\langle X_b,\nabla_{X_a}Z_r\rangle=-\frac12\langle[X_a,X_b],Z_r\rangle,\quad\langle Z_s,\nabla_{X_a}Z_r\rangle=0,$$
for any $a,b=1,\ldots,4$, $r,s={\mathcal I},{\mathcal J},{\mathcal K}$. This translates to the equation
\begin{equation}
\nabla_{X_a}Z_r=-\frac12\sum_{b=1}^4\langle[X_a,X_b],Z_r\rangle X_b,
\end{equation}
which reduces to the following identities
$$\nabla_{X_1}Z_{\mathcal I}=\frac12X_2,\quad\nabla_{X_2}Z_{\mathcal I}=-\frac12X_1,\quad\nabla_{X_3}Z_{\mathcal I}=\frac12X_4,\quad\nabla_{X_4}Z_{\mathcal I}=-\frac12X_3,$$
$$\nabla_{X_1}Z_{\mathcal J}=-\frac12X_4,\quad\nabla_{X_2}Z_{\mathcal J}=-\frac12X_3,\quad\nabla_{X_3}Z_{\mathcal J}=\frac12X_2,\quad\nabla_{X_4}Z_{\mathcal J}=\frac12X_1,$$
$$\nabla_{X_1}Z_{\mathcal K}=-\frac12X_3,\quad\nabla_{X_2}Z_{\mathcal K}=\frac12X_4,\quad\nabla_{X_3}Z_{\mathcal K}=\frac12X_1,\quad\nabla_{X_4}Z_{\mathcal K}=-\frac12X_2.$$

Therefore, it follows that the maps $J_r:{\mathcal D}\to{\mathcal D}$ defined by
$$J_r(X)=2\nabla_XZ_r,\quad r={\mathcal I},{\mathcal J},{\mathcal K},$$
are almost complex structures. Note that the equation
\begin{equation}\label{complexHtype}
\langle J_r(U_1),U_2\rangle+\langle U_1,J_r(U_2)\rangle=0
\end{equation}
holds for every $r={\mathcal I},{\mathcal J},{\mathcal K}$ and every $U_1,U_2\in\mathcal{D}$. Note in particular that equation~\eqref{complexHtype} implies that $\langle U,J_r(U)\rangle=0$ for all $U\in\mathcal{D}$.

\subsection{A variational argument}

Consider a manifold $M$ and let $\dist\subset TM$ be a distribution. A variation of a curve $\gamma:[a,b]\to M$ is a $C^2$-map $\tilde\gamma:I_1 \times I_2\to M$, where $I_1, I_2$ are open intervals, $0\in I_2$ and $\tilde\gamma(s,0)=\gamma(s)$. In what follows, we will denote $\tilde\gamma(s,\varepsilon)=\gamma_\varepsilon(s)$.

Let $W_\varepsilon$ be the vector field along $\gamma_\varepsilon$ given by
$$W_{\varepsilon}(s)=\left.\frac{\partial\gamma_\tau(s)}{\partial\tau}\right|_{\tau=\varepsilon}=\frac{\partial\gamma}{\partial\tau}(s,\varepsilon).$$
Note that the vector fields $W_\varepsilon$ and $\dot\gamma_\varepsilon$ commute
$$[W_\varepsilon,\dot\gamma_\varepsilon]=\left[\frac{\partial\gamma}{\partial\varepsilon}(s,\varepsilon),\frac{\partial\gamma}{\partial s}(s,\varepsilon)\right]=\left[\frac{\partial}{\partial\varepsilon},\frac{\partial}{\partial s}\right]\gamma(s,\varepsilon)=0.$$

A variation $\gamma_\varepsilon$ of a horizontal curve $\gamma$ is called admissible if all curves $\gamma_\varepsilon:I_1\to M$ are horizontal, $\gamma_\varepsilon(a)=\gamma(a)$ and $\gamma_\varepsilon(b)=\gamma(b)$ for all $\varepsilon\in I_2$. Observe that for an admissible variation of $\gamma$, the vector field $W_0$ vanishes at the endpoints of $\gamma$: $W_0(\gamma(a))=W_0(\gamma(b))=0$.

Let us study an admissible variation $\gamma_\varepsilon$ of a horizontal curve $\gamma$ in the case of ${\mathbf H}^1$, with the Riemannian metric defined in the previous Subsection. Since the variation is admissible, we have
$$\langle\dot\gamma_\varepsilon,Z_{\mathcal I}\rangle=\langle\dot\gamma_\varepsilon,Z_{\mathcal J}\rangle=
\langle\dot\gamma_\varepsilon,Z_{\mathcal K}\rangle=0.$$

In what follows, for an arbitrary vector field $X$ on ${\mathbf H}^1$, we will denote by $X_H$ and $X_V$ the orthogonal projections of $X$ to the horizontal distribution $\mathcal{D}\subset T\mathbf{H}^1$ and the vertical bundle $\spn\{Z_{\mathcal I},Z_{\mathcal J},Z_{\mathcal K}\}$ respectively.

The horizontality conditions $\langle\dot\gamma_\varepsilon,Z_r\rangle=0$, for $r={\mathcal I},{\mathcal J},{\mathcal K}$, yield
\begin{eqnarray*}
0=\left.\frac{d}{d\varepsilon}\right|_{\varepsilon=0}\langle\dot\gamma_\varepsilon,Z_r\rangle&=& \langle\nabla_{W_0}\dot\gamma,Z_r\rangle+\langle\dot\gamma,\nabla_{W_0}Z_r\rangle\\
&=&\langle\nabla_{\dot\gamma}W_0,Z_r\rangle+\langle\dot\gamma,\nabla_{W_{0_H}}Z_r\rangle\\
&=&\dot\gamma\langle W_0,Z_r\rangle-\langle W_0,\nabla_{\dot\gamma}Z_r\rangle+\langle\dot\gamma,J_r(W_{0_H})\rangle\\
&=&\dot\gamma\langle W_0,Z_r\rangle-\langle W_{0_H},J_r(\dot\gamma)\rangle-\langle J_r(\dot\gamma),W_{0_H}\rangle\\
&=&\dot\gamma\langle W_0,Z_r\rangle-2\langle W_{0_H},J_r(\dot\gamma)\rangle,
\end{eqnarray*}
where we have used equation~\eqref{complexHtype} and $\nabla_{Z_s}Z_r=0$.

In fact the converse statement also holds.

\begin{lemma}\label{admvar}
Let $W$ be any $C^1$ vector field along $\gamma$ such that $W(\gamma(a))=W(\gamma(b))=0$ and that satisfies
$$0=\dot\gamma\langle W,Z_r\rangle-2\langle W_H,J_r(\dot\gamma)\rangle.$$
Then there exists an admissible variation $\gamma_\varepsilon$ of $\gamma$ such that
$$\left.\frac{\partial}{\partial\varepsilon}\right|_{\varepsilon=0}\gamma(s,\varepsilon)=W.$$
\end{lemma}

\begin{proof}
Let us decompose $W=f\dot\gamma+\widetilde W$, with $\widetilde W\bot\dot\gamma$ and $f(\gamma(a))=f(\gamma(b))=0$. With this definition, we have
$$\langle W,\dot\gamma\rangle=f,\quad\langle W,J_r(\dot\gamma)\rangle=\langle\widetilde W,J_r(\dot\gamma)\rangle,\quad\langle W,Z_r\rangle=\langle\widetilde W,Z_r\rangle.$$
Observe that the term $f\dot\gamma$ will not contribute to any admissible variation, therefore we can assume that $W\bot\dot\gamma$. Let $s\in I_1$ and $\varepsilon>0$ sufficiently small. Define the mapping
$$F(s,\varepsilon)=\exp_{\gamma(s)}(\varepsilon W(s)),$$
where $\exp$ is the exponential map associated to the metric $\langle\cdot,\cdot\rangle$ of $\mathbf{H}^1$.

If $W$ is horizontal in some nonempty interval $I\subset I_1$, then $W=W_H$ and also $\langle W_H,J_r(\dot\gamma)\rangle=\frac12\dot\gamma\langle W_H,Z_r\rangle=0$. This implies $W_H=\lambda(p)\dot\gamma$, but since $W_H\bot\dot\gamma$, then $W_H=0$.

If $W(s_0)$ is not horizontal, then $F(s,\varepsilon)$ defines locally a surface which is foliated by horizontal curves and it is transversal to the horizontal distribution, since it contains curves in nonhorizontal directions. This implies there exists a $C^2$ function $g(s,\varepsilon)$ such that
$$\gamma_\varepsilon(s)=\exp_{\gamma(s)}(g(s,\varepsilon)W(s))$$
is a horizontal curve. Choosing $g$ such that $\left.\dfrac{\partial}{\partial\varepsilon}\right|_{\varepsilon=0} f(s_0,\varepsilon)=1$, we obtain an admissible variation $\gamma_\varepsilon$ of $\gamma$ with associated vector field $W$.
\end{proof}

With this result at hand, we can formulate the main theorem of this section.

\begin{theo}\label{curvS7}
Let $\gamma:[a,b]\to {\mathbf H}^1$ be a horizontal curve, parameterized by arc length. Then $\gamma$ is a critical point of the length functional if and only if there exist $\lambda_{\mathcal I},\lambda_{\mathcal J},\lambda_{\mathcal K}\in\real$ satisfying the second order differential equation
\begin{equation}\label{curvS7eq}
\nabla_{\dot\gamma}\dot\gamma-2\sum_{r={\mathcal I},{\mathcal J},{\mathcal K}}\lambda_rJ_r(\dot\gamma)=0.
\end{equation}
\end{theo}

\begin{proof}
Let $\gamma:I=[a,b]\to {\mathbf H}^1$ be a horizontal curve, parameterized by arc length, and let $\gamma_\varepsilon$ be an admissible variation of $\gamma$, with vector field $U$. The first variation of the length functional, see~\cite{CE}, is given by
\begin{equation}\label{1stvar}
\left.\frac{d}{d\varepsilon}\right|_{\varepsilon=0}L(\gamma_\varepsilon)= -\int_I\langle\nabla_{\dot\gamma}\dot\gamma,U\rangle.
\end{equation}
Suppose $\gamma$ is a critical point of the first variation, that is
$$\int_I\langle\nabla_{\dot\gamma}\dot\gamma,U\rangle=0.$$

The condition $\|\dot\gamma\|=1$ implies $\langle\nabla_{\dot\gamma}\dot\gamma,\dot\gamma\rangle=0$. Since $\langle\dot\gamma,Z_r\rangle=0$ for $r={\mathcal I},{\mathcal J},{\mathcal K}$, we have
\begin{eqnarray*}
0=\dot\gamma\langle\dot\gamma,Z_r\rangle&=&\langle\nabla_{\dot\gamma}\dot\gamma,Z_r\rangle+\langle\dot\gamma,\nabla_{\dot\gamma}Z_r\rangle\\
&=&\langle\nabla_{\dot\gamma}\dot\gamma,Z_r\rangle+\langle\dot\gamma,J_r(\dot\gamma)\rangle\\
&=&\langle\nabla_{\dot\gamma}\dot\gamma,Z_r\rangle.
\end{eqnarray*}
Therefore, counting dimensions
\begin{equation}
\nabla_{\dot\gamma}\dot\gamma=\displaystyle{\sum_{r={\mathcal I},{\mathcal J},{\mathcal K}}g_r(\gamma)J_r(\dot\gamma)}.
\end{equation}
In order to prove that the functions $g_r$ are constant, fix three $C^1$ functions $f_{\mathcal I},f_{\mathcal J},f_{\mathcal K}:I\to\real$ such that $f_r(a)=f_r(b)=0$ and $\ds{\int_If_r}=0$ for $r={\mathcal I},{\mathcal J},{\mathcal K}$. Consider a vector field $\tilde U$ such that $\tilde{U}_H=\sum_{r={\mathcal I},{\mathcal J},{\mathcal K}}f_rJ_r(\dot\gamma)$ and $\langle\tilde{U},Z_r\rangle(s)=2\int_a^sf_r(t)dt$. We claim that $\tilde{U}$ satisfies
$$\dot\gamma\langle\tilde{U},Z_r\rangle=2\langle\tilde{U}_H,J_r(\dot\gamma)\rangle,$$
for ${r={\mathcal I},{\mathcal J},{\mathcal K}}$. To see this, observe that
$$\dot\gamma\langle\tilde{U},Z_r\rangle=\frac{d}{ds}\left(2\int_a^sf_r(t)dt\right)=2f_r(s)$$
and also
$$2\langle\tilde{U}_H,J_r(\dot\gamma)\rangle=2\left\langle\sum_{s={\mathcal I},{\mathcal J},{\mathcal K}}f_sJ_s(\dot\gamma),J_r(\dot\gamma)\right\rangle=
2f_r(s).$$

Thus, by Lemma~\ref{admvar}, we can conclude that $\tilde{U}$ is a vector field for an admissible variation of $\gamma$. By the variational identity~\eqref{1stvar}, we obtain the equality
$$0=\int_I\langle\nabla_{\dot\gamma}\dot\gamma,\tilde{U}\rangle=\sum_{r={\mathcal I},{\mathcal J},{\mathcal K}}\int_If_r\langle\nabla_{\dot\gamma}\dot\gamma,J_r(\dot\gamma)\rangle,$$
which is valid for any three functions with mean zero. This implies that the functions $\langle\nabla_{\dot\gamma}\dot\gamma,J_r(\dot\gamma)\rangle$ are constant, and thus we obtain equation~\eqref{curvS7eq}, for suitable constants $\lambda_{\mathcal I},\lambda_{\mathcal J},\lambda_{\mathcal K}\in\real$.

Conversely, let us assume that $\gamma$ is a horizontal curve, such that $\|\dot\gamma\|=1$ and it satisfies the differential equation~\eqref{curvS7eq}, for some $\lambda_{\mathcal I},\lambda_{\mathcal J},\lambda_{\mathcal K}\in\real$. We need to show that
$$\int_I\langle\nabla_{\dot\gamma}\dot\gamma,U\rangle=0$$
for any $C^1$-smooth vector field $U$, vanishing at the endpoints of $\gamma$ and sa\-tisfying
$$\dot\gamma\langle U,Z_r\rangle=2\langle U_H,J_r(\dot\gamma)\rangle,$$
where $r={\mathcal I},{\mathcal J},{\mathcal K}$.

Let us write $U=U_H+U_V=\displaystyle{U_H+\sum_{r={\mathcal I},{\mathcal J},{\mathcal K}}g_rZ_r}$, where $g_r(\gamma(a))=g_r(\gamma(b))=0$, then
\begin{eqnarray*}
\int_I\langle\nabla_{\dot\gamma}\dot\gamma,U\rangle&=&-2\sum_{r={\mathcal I},{\mathcal J},{\mathcal K}}\lambda_r\int_I\langle J_r(\dot\gamma),U\rangle=
-2\sum_{r={\mathcal I},{\mathcal J},{\mathcal K}}\lambda_r\int_I\langle J_r(\dot\gamma),U_H\rangle\\
&=&-\sum_{r={\mathcal I},{\mathcal J},{\mathcal K}}\lambda_r\int_I\dot\gamma\langle U,Z_r\rangle=
-\sum_{r={\mathcal I},{\mathcal J},{\mathcal K}}\lambda_r\int_I\dot\gamma\langle U_V,Z_r\rangle\\
&=&-\sum_{r={\mathcal I},{\mathcal J},{\mathcal K}}\lambda_r\int_I\dot\gamma(g_r)=-\sum_{r={\mathcal I},{\mathcal J},{\mathcal K}}\lambda_r\int_a^b\frac{d}{dt}(g_r(\gamma(t)))=0.
\end{eqnarray*}
\end{proof}

\section{The intrinsic sub-Laplacian for $S^7$ with growth vector $(6,1)$}

In \cite{ABGR} the authors presented an intrinsic form of the sub-Laplacian, by means of Popp's measure $\mu_{sR}$, introduced in \cite{M}. The aim of this section is to construct this differential operator for the case of $S^7$ endowed with the contact distribution, introduced in Section 3.

\subsection{Construction of the intrinsic sub-Laplacian}

Let $(M,\dist,\langle\cdot,\cdot\rangle_{sR})$ be a sub-Riemannian manifold, where $\dist$ is a regular distribution. The basic idea is to define the intrinsic sub-Laplacian $\Delta_{sR}f$ of a function $f:M\to\real$ of class $C^2$, in analogy to the Riemannian case. To do this, let us define the horizontal gradient $\nabla_{sR}f$ by the equation
\begin{equation}
\langle\nabla_{sR}f(p),v\rangle_{sR}=d_pf(v),
\end{equation}
and the sub-Riemannian divergence ${\rm{div}}_{sR}X$ of a horizontal vector field $X$ by
\begin{equation}
{\rm{div}}_{sR}X\mu_{sR}=L_X\mu_{sR},
\end{equation}
where $\mu_{sR}\in\bigwedge^n(T^*M)$ is a fixed non-vanishing $n-$form, known as Popp's volume form, and $L_X$ denotes the Lie derivative in the direction of $X$. The intrinsic sub-Laplacian is given by
\begin{equation}\label{defLap}
\Delta_{sR}f={\rm{div}}_{sR}(\nabla_{sR}f).
\end{equation}
For full details about its construction, see~\cite{ABGR,M}.

\paragraph{\bf Remark:} In the Riemannian case this definition coincides with the classical definition of the Laplacian, see for example \cite{Ro}. As pointed out in~\cite{ABGR}, the regularity hypothesis over the distribution cannot be avoided since for example, in the case of the Grushin plane, the operator~\eqref{defLap} is not hypoelliptic.

Let $\{X_1,\ldots,X_k\}$ be a local orthonormal basis of $\dist\subset TM$ and consider the corresponding dual basis $\{dX_1,\ldots,dX_k\}$. It is possible to find vector fields $\{X_{k+1},\ldots,X_n\}$ such that $\spn\{X_1,\ldots,X_n\}=TM$ and such that Popp's volume form is locally given by
\begin{equation}
\mu_{sR}=dX_1\wedge \ldots\wedge dX_k\wedge dX_{k+1}\wedge\ldots\wedge dX_n.
\end{equation}

In this setting, the sub-Laplacian $\Delta_{sR}f$ can be written explicitly as
\begin{equation}\label{expsubLap}
\Delta_{sR}f=\sum_{r=1}^k\left(L_{X_r}^2f+L_{X_r}f\sum_{s=1}^ndX_s([X_r,X_s])\right).
\end{equation}

\subsection{Examples}

The case of the intrinsic sub-Laplacian for $S^3$ is implied by the following result, characterizing Popp's volume form for contact manifolds of dimension 3.

\begin{prop}[\cite{ABGR,M}]\label{contpopp}
Let $M$ be a three dimensional orientable contact manifold with a sub-Riemannian metric defined on its contact distribution. Let $\{X_1,X_2\}$ a local orthonormal frame for its contact distribution. Let $X_3=[X_1,X_2]$ and $\{dX_1,dX_2,dX_3\}$ be the dual basis to $\{X_1,X_2,X_3\}$. Then the form $dX_1\wedge dX_2\wedge dX_3$ is an intrinsic volume form.
\end{prop}

In particular, for the sphere $S^3$ endowed with the contact distribution generated by the globally defined vector fields~\eqref{distS3}, with commutator
$$[X,Y](x)=2V(x)=2(-x_1 \partial_{x_0}+x_0 \partial_{x_1}-x_3\partial_{x_2}+x_2\partial_{x_3}),$$
Popp's volume form, as constructed above, is $2dX\wedge dY\wedge dV$, and the intrinsic sub-Laplacian is given by
$$\Delta_{sR}f=(X^2+Y^2)f.$$

In general, we can extend the previous result to construct locally Popp's volume form over contact manifolds of arbitrary dimension. Let $M$ be a contact manifold of dimension $2n+1$, with contact form $\omega$ and contact distribution $\xi=\ker\omega$. The distribution $\xi$ is bracket generating of step two, see~\cite{GM}. Assume that $M$ has a Riemannian metric $g$ such that, in a neighborhood of each $p\in M$, there is an orthonormal basis $B=\{v_1,\ldots,v_{2n},v_{2n+1}\}$ for $T_pM$ satisfying $\xi_p=\spn\{v_1,\ldots,v_{2n}\}$. Following the construction in \cite{M} we have that Popp's volume form in this case is given locally by
\begin{equation}
\mu_{sR}=\pi_1\wedge\ldots\wedge\pi_{2n+1},
\end{equation}
where $B^*=\{\pi_1,\ldots,\pi_{2n+1}\}$ is the dual basis for $B$.

In the case of the contact structure of $S^7$, let us consider the vector fields $X_1,\ldots,X_7$ presented in the Appendix. Since the vector fields $X_1$ and $V_4$ from equation~\eqref{defV} coincide, the contact distribution on $S^7$ introduced in Section 3 corresponds to
$$\dist=\ker\omega=\spn\{X_2,\ldots,X_7\}.$$
In this context we have the following

\begin{theo}
Let $\dist$ be the contact distribution for $S^7$ and $\langle\cdot,\cdot\rangle_{sR}$ the restriction of the usual Riemannian metric in $\real^8$ to $\dist$. Then the intrinsic sub-Laplacian of $(S^7,\dist,\langle\cdot,\cdot\rangle_{sR})$ is given by the sum of squares
$$\Delta_{sR}=\sum_{a=2}^7X_a^2.$$
\end{theo}

\begin{proof}
The construction of Popp's measure leads to the globally defined $n$ form
$$\mu_{sR}=dX_1\wedge\ldots\wedge dX_7,$$
which is precisely the Riemannian volume form of $S^7$. Simple calculations show that
\begin{equation}\label{ortcomm}
dX_b([X_a,X_b])=\langle X_b,[X_a,X_b]\rangle_{sR}=0,\quad a=2,\ldots,7\quad b=1,\ldots7.
\end{equation}

The theorem follows from formula~\eqref{expsubLap}.
\end{proof}

\paragraph{\bf Remark:} A complete list of the commutators $[X_a,X_b]$, for $a<b$, can be found in~\cite[Section 8]{GM}. This list can be used to check equation~\eqref{ortcomm} directly.

\section{Heat operator for $S^7$ with growth vector $(6,1)$}

The aim of this section is to show that the above constructed operator $\Delta_{sR}$ commutes with the operator $X_1^2$. A similar observation was exploited to study the heat operator for the sub-Riemannian structure of $SU(2)\cong S^3$ in~\cite{BB}.

The main result of this Section is formulated as follows.

\begin{theo}\label{commy1}
The operators $\Delta_{sR}$ and $X_1^2$ commute.
\end{theo}

\begin{proof}
Let us introduce the following change of coordinates for $S^7$:
\begin{equation}
\begin{array}{ccc}
x_0+ix_1&=&e^{i\xi_1}\cos\eta_1\cos\psi\\
x_2+ix_3&=&e^{i\xi_2}\sin\eta_1\cos\psi\\
x_4+ix_5&=&e^{i\xi_3}\cos\eta_2\sin\psi\\
x_6+ix_7&=&e^{i\xi_4}\sin\eta_2\sin\psi
\end{array}
\end{equation}

By the chain rule, the symbol of the sub-Laplacian $\Delta_{sR}=X_2^2+\ldots+X_7^2$ is a quadratic form with matrix
$$\left(\begin{array}{ccccccc}h_1(\eta_1,\psi)&-1&-1&-1&0&0&0\\-1&h_2(\eta_1,\psi)&-1&-1&0&0&0\\-1&-1&h_3(\eta_2,\psi)&-1&0&0&0\\
-1&-1&-1&h_4(\eta_2,\psi)&0&0&0\\0&0&0&0&\sec^2\psi&0&0\\0&0&0&0&0&\csc^2\psi&0\\0&0&0&0&0&0&1\end{array}\right),$$
where the coefficient functions $h_1,h_2,h_3$ and $h_4$ are given by

$$h_1(\eta_1,\psi)=-\frac{\sec^2(\eta_1)\sec^2(\psi)}8\Big(-6+2\cos(2\eta_1)+\cos(2(\eta_1-\psi))+$$
$$+2\cos(2\psi)+\cos(2(\eta_1+\psi))\Big),$$

$$h_2(\eta_1,\psi)=\frac{\csc^2(\eta_1)\sec^2(\psi)}8\Big(6+2\cos(2\eta_1)+\cos(2(\eta_1-\psi))-$$
$$-2\cos(2\psi)+\cos(2(\eta_1+\psi))\Big),$$

$$h_3(\eta_2,\psi)=\frac{\sec^2(\eta_2)\csc^2(\psi)}8\Big(6-2\cos(2\eta_2)+\cos(2(\eta_2-\psi))+$$
$$+2\cos(2\psi)+\cos(2(\eta_2+\psi))\Big),$$

$$h_4(\eta_2,\psi)=-\frac{\csc^2(\eta_2)\csc^2(\psi)}8\Big(-6-2\cos(2\eta_2)+\cos(2(\eta_2-\psi))-$$
$$-2\cos(2\psi)+\cos(2(\eta_2+\psi))\Big).$$

Observe that $h_1,\ldots,h_4$ are independent of $\xi_1,\ldots,\xi_4$. On the other hand, the vector field $X_1$, written in the new coordinates, becomes
$$X_1=\partial_{\xi_1}+\partial_{\xi_2}+\partial_{\xi_3}+\partial_{\xi_4}.$$

Since the coefficients of $\Delta_{sR}$ are independent of the variables $\xi_1,\xi_2,\xi_3$ and $\xi_4$, it is clear that the operators $\Delta_{sR}$ and $X_1$ commute. The Theorem follows.
\end{proof}

Let us denote by $e^{-t\Delta_{sR}}$ the semigroup of operators acting on $L^2_{\mu_{sR}}$, with infinitesimal generator $\Delta_{sR}$. The operator $e^{-t\Delta_{sR}}$ is known as the sub-Riemannian heat operator. As a consequence of Theorem~\ref{commy1}, we get the announced result.

\begin{coro}\label{heat}
Denoting by $\Delta_{S^7}$ the Laplace-Beltrami operator in $S^7$ with respect to the usual Riemannian structure, we have that
$$e^{-t\Delta_{S^7}}=e^{-t(\Delta_{sR}+X_1^2)}=e^{-t\Delta_{sR}}e^{-tX_1^2}.$$
\end{coro}

\begin{proof}
Since $\Delta_{S^7}=\Delta_{sR}+X_1^2$, we have by the commutativity of the operators
\begin{equation}
e^{-t\Delta_{S^7}}=e^{-t(\Delta_{sR}+X_1^2)}=e^{-t\Delta_{sR}}e^{-tX_1^2},
\end{equation}
yielding to the stated result.
\end{proof}

The theory of unbounded operators allows us to rephrase the result in Corollary~\ref{heat} as:

\begin{coro}
The sub-Riemannian heat operator $e^{-t\Delta_{sR}}$ is given by
$$e^{-t\Delta_{sR}}=e^{-t\Delta_{S^7}}e^{tX_1^2}.$$
\end{coro}

\section{Appendix: Tangent vector fields to $S^7$}

Octonion multiplication induces the following orthonormal basis of $TS^7$ with respect to the restriction of the inner product $\langle\cdot,\cdot\rangle$ from $\mathbb R^8$ to the tangent space $T_pS^7$ at each $p\in S^7$.
\begin{eqnarray*}
X_1(x)&=&-x_1\partial_{x_0}+x_0\partial_{x_1}-x_3\partial_{x_2}+
x_2\partial_{x_3}-x_5\partial_{x_4}+x_4\partial_{x_5}-
x_7\partial_{x_6}+x_6\partial_{x_7}
\\
X_2(x)&=&-x_2\partial_{x_0}+x_3\partial_{x_1}+x_0\partial_{x_2}-
x_1\partial_{x_3}-x_6\partial_{x_4}+x_7\partial_{x_5}+
x_4\partial_{x_6}-x_5\partial_{x_7}
\\
X_3(x)&=&-x_3\partial_{x_0}-x_2\partial_{x_1}+x_1\partial_{x_2}+
x_0\partial_{x_3}+x_7\partial_{x_4}+x_6\partial_{x_5}-
x_5\partial_{x_6}-x_4\partial_{x_7}
\\
X_4(x)&=&-x_4\partial_{x_0}+x_5\partial_{x_1}+x_6\partial_{x_2}-
x_7\partial_{x_3}+x_0\partial_{x_4}-x_1\partial_{x_5}-
x_2\partial_{x_6}+x_3\partial_{x_7}
\\
X_5(x)&=&-x_5\partial_{x_0}-x_4\partial_{x_1}-x_7\partial_{x_2}-
x_6\partial_{x_3}+x_1\partial_{x_4}+x_0\partial_{x_5}+
x_3\partial_{x_6}+x_2\partial_{x_7}
\\
X_6(x)&=&-x_6\partial_{x_0}+x_7\partial_{x_1}-x_4\partial_{x_2}+
x_5\partial_{x_3}+x_2\partial_{x_4}-x_3\partial_{x_5}+
x_0\partial_{x_6}-x_1\partial_{x_7}
\\
X_7(x)&=&-x_7\partial_{x_0}-x_6\partial_{x_1}+x_5\partial_{x_2}+
x_4\partial_{x_3}-x_3\partial_{x_4}-x_2\partial_{x_5}+
x_1\partial_{x_6}+x_0\partial_{x_7}.
\end{eqnarray*}

\end{document}